\documentclass[11pt]{amsart}

\usepackage{amssymb,amsmath,amsthm,amscd,mathrsfs,graphicx}
\usepackage[cmtip,all]{xy}
\usepackage{pb-diagram,pb-xy} 
\usepackage{amscd}

\begin{document}

\theoremstyle{plain} \numberwithin{equation}{section}
\newtheorem*{unlabeledtheorem}{Theorem}
\newtheorem{theorem}{Theorem}[section]
\newtheorem{corollary}[theorem]{Corollary}
\newtheorem{conjecture}{Conjecture}
\newtheorem{lemma}[theorem]{Lemma}
\newtheorem{proposition}[theorem]{Proposition}
\newtheorem{convention}[theorem]{}
\newtheorem*{unlabeledproposition}{Proposition}
\newtheorem{theoremalpha}{Theorem}
\renewcommand{\thetheoremalpha}{\Alph{theoremalpha}}
\newtheorem{propositionalpha}[theoremalpha]{Proposition}
\newtheorem{corollaryalpha}[theoremalpha]{Corollary}

\newtheorem{fact}{Fact}

\theoremstyle{definition}
\newtheorem{definition}[theorem]{Definition}
\newtheorem{finalremark}[theorem]{Final Remark}
\newtheorem{remark}[theorem]{Remark}
\newtheorem*{unlabeledremark}{Remark}
\newtheorem{example}[theorem]{Example}
\newtheorem{question}{Question} 
\newtheorem*{warning}{Warning}

\title[RST for integral curves]{A Riemann singularity theorem for integral curves}
\author[Casalaina-Martin]{Sebastian Casalaina-Martin}
\address{University of Colorado at Boulder, Department of Mathematics,  Campus Box 395,
Boulder, CO 80309-0395, USA}
\email{casa@math.colorado.edu}

\author[Kass]{Jesse Leo Kass}
\address{University of Michigan, Department of Mathematics,  530 Church Street, Ann Arbor, MI 48109}
\email{jkass@umich.edu}
\thanks{The second author was partially supported by NSF grant DMS-0502170}

\date{\today}

\begin{abstract}
We prove results generalizing the classical Riemann Singularity Theorem to the case of integral, singular curves.  The main result is a computation of the multiplicity of the theta divisor of an integral, nodal curve at an arbitrary point.  We also suggest a general formula for the multiplicity of the theta divisor of a singular, integral curve at a point and present some evidence that this formula should hold.   Our results give a partial answer to a question posed by Lucia Caporaso in a recent paper.
\end{abstract}

\maketitle
\bibliographystyle{amsplain}

\section*{Introduction}
In this article, we study the local geometry of the theta divisor of an integral curve with the goal of extending the Riemann Singularity Theorem.   Motivated in part by work of Caporaso \cite{caporaso07}, we compute the multiplicity of the theta divisor of an integral, nodal curve at a point and then use this to determine the singular locus of the theta divisor.  We also suggest a formula that would extend our results to more general curves and present some evidence that this formula should hold. 

	To motivate our results, recall that associated to a non-singular curve $X/k$ of genus $g\ge2$ is its Jacobian variety $J_{X/k}^{g-1}$ parametrizing degree $g-1$ line bundles on $X$.  The locus   $\Theta$ corresponding to line bundles that admit a non-zero global section is an ample divisor known as the theta divisor.  The Riemann Singularity Theorem  states that, if $x$ is a point of $\Theta$ that corresponds to a line bundle $L$, then $\operatorname{mult}_{x} \Theta  = h^{0}(X,L)$.
	
	For a singular curve $X/k$, the moduli space of degree $g-1$  line bundles is typically non-complete, and a   completion of this space is given by the moduli space of rank $1$, torsion-free sheaves, denoted $\bar J^{g-1}_{X/k}$.  There is a natural analogue $\Theta$ of the classical theta divisor that lies on $\bar{J}_{X/k}^{g-1}$.  The main result of this paper is the following generalization of the Riemann Singularity Theorem:

\begin{theoremalpha}\label{teoC}
		Suppose that $X/k$ is an integral curve with at worst planar singularities.  Let $x$ be a point of the theta divisor $\Theta$ corresponding to a rank $1$, torsion-free sheaf $I$.  If the sheaf $I$ fails to be locally free at $n$ nodes and no 	
		other points, then the multiplicity and order of vanishing of $\Theta$ at $x$ satisfy the equation
\begin{displaymath}
	\operatorname{mult}_{x}\Theta = \left(\operatorname{mult}_x \bar{J}_{X/k}^{g-1}\right)\cdot \operatorname{ord}_x\Theta =2^{n}  \cdot h^{0}(X,I).
\end{displaymath}
\end{theoremalpha}
Recall that $\operatorname{ord}_{x}\Theta$ is defined to be the largest power of the maximal ideal of $x$ that contains a local equation for $\Theta$.  It is known (Corollary \ref{CorMultJac})  that $\operatorname{mult}_x\bar J^{g-1}_{X/k}=2^n$, and so part of the statement is that $\operatorname{ord}_x\Theta=h^0(X,I)$.  A geometric interpretation of the number $h^{0}(X,I)$ in terms of a linear system is described in \S \ref{AbelSection}.  The case of the theorem where $I$ is a line bundle is due to Kempf \cite{kempf70} using different methods.

When $X$ is nodal, this result allows us to immediately determine the singular locus of $\Theta$.  Let us write $\partial \Theta$ for the locus in $\Theta$ that corresponds to sheaves that fail to be locally free and $W^1_{g-1}$ for the locus of sheaves $I$ satisfying $h^{0}(X,I) \ge 2$.  An immediate consequence of the previous theorem is that $\Theta_{\text{sing}} = \partial \Theta \cup W^{1}_{g-1}$. This statement can be proven in greater generality, and this is done in \S\ref{SectionPfA}.

Our proof of Theorem \ref{teoC} is a generalization of the proof of the classical Riemann Singularity Theorem given by Friedman and the first author in \cite{yano}.  The multiplicity is computed by taking test arcs on the compactified Jacobian and studying the intersection of the arc with the theta divisor.  By general formalism, constructing appropriate arcs is equivalent to constructing certain 1-parameter families of rank 1, torsion-free sheaves, and such families can be constructed directly (\S \ref{ThetaSection}).  There are some obstacles to applying the methods of \cite{yano} that do not appear when working with non-singular curves or, more precisely, when working within the moduli space of line bundles on a possibly singular curve, and part of this paper is devoted to extending the tools from that paper.

In light of Theorem \ref{teoC}, it is natural to ask what one can say about the multiplicity of the theta divisor of an integral curve with arbitrary  planar singularities.  To be precise, if $X/k$ is an integral curve with at worst planar singularities and $x$ is a point of the theta divisor that corresponds to a sheaf $I$, then it would be interesting to know if the equalities
\begin{equation}\label{ConjA}
	\operatorname{mult}_{x}\Theta = \left( 
	\operatorname{mult}_{x}\bar{J}^{g-1}_{X/k} \right) \cdot \operatorname{ord}_{x} \Theta \ \ 
	\textnormal{ and }\ \ \operatorname{ord}_x\Theta=h^0 (X,I)
\end{equation} 
still held.  Note that $\ge$ always holds in the first formula (e.g. \eqref{eqnmat})  and $\le$ in the second (Proposition \ref{ThetaFuncProp}, \eqref{eqnbasicarc}).
While we can not answer this question here, we are able to prove some further results suggesting this may be the case.
 We refer the reader to \S\ref{SectionPfA} for the statements and proofs.
 
 The literature on the classical Riemann Singularity Theorem is vast, and a guide to it can be found in the book \cite{ACGH}.  Particularly relevant to this article are the papers of Kempf   \cite{kempf70,kempf73} and Beauville  \cite{beauville77}.    The focus of study in these papers is on the local structure of $\Theta$ at a point corresponding to a line bundle, although integral curves with arbitrary singularities, as well as reducible curves, are also considered.

There are some papers that treat the case of a point that does not correspond to a line bundle.  Results about the local geometry of $\Theta$ at such a point were proven by Bhosle and Parameswaran  and can be found in  \cite[Theorem 5.4]{bp}.  Their goals are somewhat different from ours, and their results are similar to Proposition \ref{teoA} of this paper.  Samuel Grushevsky has suggested to the authors it may be possible to derive certain cases of Theorem \ref{teoC} using results in  \cite{mk}.

The results of this paper suggest a few avenues for further research.  In one direction, it would be desirable to further study the local geometry of the theta divisor of an integral, nodal curve.  In this paper we compute the multiplicity $\operatorname{mult}_{x}\Theta$; i.e., the degree of the tangent cone $\mathscr{T}_{x}\Theta$, and  one could ask for a more precise description of this cone.  It would also be interesting to relate geometric properties of Brill-Noether loci to geometric properties of the curve, as there is a body of such results  (e.g. Martens' Theorem) for non-singular curves.  The two authors will address  these topics in a subsequent paper.

In a different direction, one could ask for an extension of Theorem \ref{teoC} to more general curves.  We have already posed such a question for integral curves with planar singularities, but one can also ask the question for reducible, nodal curves;  Caporaso's original question concerned the theta divisor of a stable curve.  There are some complications that arise in trying to extend the results of this paper to reducible curves, and this question is currently being investigated by  Filippo Viviani and the two authors.

\subsection*{Acknowledgements}
The authors are very grateful to Steven Kleiman  for many useful and encouraging conversations about integral curves, compactified Jacobians, and their theta divisors.  His comments on an early draft of this document were extremely helpful.  In particular, he suggested that \S \ref{MultiplicitySection} be written in the language of commutative algebra; this not only allowed the authors to improve the exposition, but also to discover an error (which has since been fixed).   We would also like to take the opportunity to thank   Lucia Caporaso and Eduardo Esteves for many productive email exchanges concerning the literature on compactified Jacobians.   Robert Lazarsfeld and the two referees provided useful feedback that greatly improved the exposition.  Work on this paper was done while the second author was a student of Joseph Harris.  Both authors would like to thank him for sharing his knowledge of compactified Jacobians and other topics.  The first author would also like to thank Harvard University and the MSRI for their hospitality during the preparation of this paper.

\section{Conventions}  

\begin{convention}
	We will work over a fixed algebraically closed field $k$.  All \textbf{schemes} are $k$-schemes and all \textbf{morphisms} are implicitly assumed to respect the $k$-structure. 
\end{convention}

\begin{convention}
	The term \textbf{point} with no adjectives attached means a $k$-valued point.  For the schemes with which we will be making point-wise arguments, the set of $k$-valued points is always in canonical bijection with the set of closed points. 
\end{convention}

\begin{convention}
	A \textbf{variety} is an integral $k$-scheme that is separated and of finite type over $k$.  We say that a variety is \textbf{non-singular} if it is smooth over $k$.
\end{convention}

\begin{convention}
	A \textbf{curve} is a proper, connected $k$-scheme of pure dimension $1$.
	A \textbf{node} on a curve $X$ is  a  point $p$ on $X$ such that the completion of the local ring  at $p$ is isomorphic to $k[[x,y]]/(x y)$. We say that a singularity of a curve is \textbf{planar} if the Zariski tangent space at that point is 2-dimensional.
\end{convention}

\begin{convention}
	We follow the convention that a \textbf{projective space} $\mathbb{P}V$ parametrizes hyperplanes in $V$. 
\end{convention}

\section{Compactified Jacobians of Integral Curves}  \label{Sec: ComJac}

We begin by reviewing the basic theory  of compactified Jacobians.  We give a precise definition and then recall some basic results concerning their geometry.  Most of the results in this section are known to experts, and the primary purpose of this section is to fix notation. 

Our exposition is based on the theory developed by Altman and Kleiman in \cite{altman80}.  D'Souza's paper \cite{dSouza} develops some of the same theory using different techniques.  
While not used here, we point out that for (possibly reducible) nodal curves there are other related theories for  compactifying Jacobians.  An approach using balanced line bundles on semi-stable models is developed in Caporaso~\cite{caporaso07, caporaso08} and Melo~\cite{melo07}.   The article Alexeev~\cite{alexeev04} discusses compactified Jacobians in the broader context of degenerations of Abelian varieties.
The relationship among these spaces is discussed in 
Pandharipande~\cite[\S10]{pandharipande96} and Alexeev~\cite{alexeev04}; we refer the reader to these sources for more details.

\subsection{Preliminaries}
The basic definition is the following.

\begin{definition}
	A {\bf rank $1$, torsion-free sheaf} $I$ on an integral curve $X/k$ is a coherent sheaf $I$ on $X$ that is generically isomorphic to the structure sheaf $\mathscr{O}_{X}$ and has the property that a non-zero local section of $\mathscr{O}_{X}$ does not kill a non-zero local section of $I$.
\end{definition}

For the proof of  Proposition \ref{ArcExistLemma}, we need the following lemma about these sheaves.
\begin{lemma} \label{Lemma: KillSections}
	Suppose that $I$ is a rank 1, torsion-free sheaf on an integral curve $X$.  If $h^{0}(X,I) \ge d$, then $h^{0}(X,I(-D)) = h^{0}(X,I)-d$ for $D$ a general effective degree $d$ divisor.  Similarly, if $h^{1}(X,I) \ge d$, then $h^{1}(X,I(D)) = h^{1}(X,I) - d$ for $D$ a general degree $d$ divisor.
\end{lemma}
\begin{proof}
	By induction, we just need to handle the case of $d=1$ and show that one suitable divisor $D$ exists.  Consider first the claim concerning global sections.  Pick a non-zero element $\sigma \in H^{0}(X,I)$.  Away from the nodes of $X$, the sheaf $I$ is a line bundle and so the zero locus of $\sigma$ is 0-dimensional. Now choose  $p \in X$ that does not lie in this zero locus and is not a node.  We have that $H^{0}(X, I(-p)) \ne H^{0}(X,I)$ since $\sigma \notin H^{0}(X, I(-p))$.  On the other hand, the codimension of $H^{0}(X,I(-p))$ in $H^{0}(X,I)$ is at most $1$ since the quotient space injects into the 1-dimensional stalk $I \otimes k(p)$.  
	
	To handle the case of $H^{1}(X,I)$, we use the duality $$H^{1}(X,I) = \operatorname{Hom}(I,\omega)^{\vee}.$$  Here the sheaf $\omega$ is the dualizing sheaf.  This sheaf is known to be rank 1 and torsion-free (\cite[\S6.5]{altman80}), so 
a non-zero element of $\operatorname{Hom}(I,\omega)$ vanishes on a 0-dimensional set.  We may now proceed as in the case of $H^{0}(X, I)$.  This completes the proof.
\end{proof}

We now define families of rank 1, torsion-free sheaves.
\begin{definition}
Suppose that $T$ is an arbitrary $k$-scheme.  A {\bf $T$-relatively rank $1$, torsion-free sheaf} is a finitely presented $\mathscr{O}_{X \times T}$-module $\mathcal{I}$ on $X \times T$ that is $T$-flat and has  fibers that are rank $1$, torsion-free sheaves in the sense defined previously. 
\end{definition}

In plain English, a $T$-relatively rank $1$, torsion-free sheaf is a flat family of rank $1$, torsion-free sheaves that is parametrized by $T$.  We will frequently abuse language and say that ``$\mathcal{I}$ is a relatively  rank 1, torsion-free sheaf," leaving the scheme $T$ implicit.  

It is convenient to single out certain families of rank $1$, torsion-free sheaves.  

\begin{definition} \label{dfntriv}
	Suppose that $\mathcal{I}$ is a $T$-relatively rank $1$, torsion-free sheaf on $X \times T$.  Let $p: X \times T \rightarrow T$ and $q: X \times T \rightarrow X$ denote the projection maps.  We say that $\mathcal{I}$ is {\bf trivial} with fiber $I$ if it is isomorphic to $q^{*}(I)$ for some rank $1$, torsion-free sheaf $I$ on $X$.
We say that $\mathcal{I}$ is {\bf iso-trivial} with fiber $I$ if there exists a line bundle $M$ on $T$ such that $\mathcal{I} \otimes p^{*}(M)$ is trivial with fiber $I$.
	We say that $\mathcal{I}$ is {\bf locally trivial} if there is an open cover $\{U_i\}$ of $X$ such that the restriction $\mathcal{I}|_{U_i\times T}$ is trivial.
\end{definition}

One can show that the condition of being iso-trivial with fiber $I$ is equivalent to the condition that the fibers of $\mathcal{I}$ are all abstractly isomorphic to $I$. We use the notion of iso-trivial to define a notion of equivalence between relatively rank $1$, torsion-free sheaves.

\begin{definition}
If $T$ is a given $k$-scheme, then we define the equivalence relation $\approx$ on $T$-relatively rank $1$, torsion-free sheaves to be the equivalence relation generated by requiring 
$\mathcal{I}$ is equivalent to $\mathcal{I} \otimes q^{*}(M)$ for every $T$-relatively rank 1, torsion-free sheaf $\mathcal{I}$ and every line bundle $M$ on $T$.
\end{definition}
In other words, $\approx$ is defined so that an iso-trivial family with fiber $I$ is equivalent to the trivial family with fiber $I$.  We will  primarily be interested in studying  $T$-relatively rank 1, torsion-free sheaves when $T$ is equal to the spectrum of either $k$ or $k[[t]]$.   In these cases, every line bundle on $T$ is trivial, so two $T$-relatively rank 1, torsion-free sheaves  are equivalent if and only if they are abstractly isomorphic as $\mathscr{O}_{X \times T}$-modules.

The degree $d$ {\bf compactified Jacobian} $\bar{J}_{X/k}^{d}$ is defined to be a scheme that represents a certain functor.

\begin{definition}
	The degree $d$ {\bf compactified Jacobian functor} $\bar{J}_{X/k}^{d}$ is defined by:
	\begin{displaymath}
		\bar{J}_{X/k}^{d}(T) =  \{ \text{fiber-wise degree } d, T\text{-rel. rank $1$, torsion-free sheaves} \} / \approx.
	\end{displaymath}
\end{definition}

	If we were working over a base that was more complicated than the spectrum of an algebraically closed field, then the above definition would need to be modified.  We will not pursue this issue here, but the relevant modifications can be found in  \cite[\S 5]{altman80}.

The basic representability theorem (\cite[Theorems 8.1, 8.5]{altman80}) states that, without further hypotheses, the compactified Jacobian functor $\bar{J}_{X/k}^{d}$ can be represented by a projective scheme.   We call this scheme the {\bf compactified Jacobian}.  

By general formalism, there is a universal family of rank 1, torsion-free sheaves $\wp$ on $\bar{J}_{X/k}^{d} \times X \rightarrow \bar{J}_{X/k}^{d}$ called the {\bf Poincar\'{e} bundle}.  It is not uniquely determined, but  any two Poincar\'{e} bundles are equivalent as relatively rank $1$, torsion-free sheaves.

\subsection{Geometry of the compactified Jacobian}\label{remkleppekass}
We now turn our attention to describing the geometry of the schemes $\bar{J}_{X/k}^{d}$.  For an arbitrary integral curve $X/k$, the geometry of the compactified Jacobian is difficult to describe.  The assumption that $X$ has at worst planar singularities, however, implies that the compactified Jacobian is a reasonably well-behaved algebro-geometric object.  More precisely, the compactified Jacobian of such a curve is a $g$-dimensional, local complete intersection variety  (\cite[Proposition 3]{Irred1977}).  
Under the same assumptions, Kleppe determined the singular locus of $\bar{J}_{X/k}^{d}$ in his (unpublished) thesis \cite{kleppe81}; the singular locus is precisely the locus of points that correspond to sheaves that fail to be locally free.  The analogous result for the Hilbert scheme can be found in Brian{\c{c}}on, Granger, and Speder  \cite{briancon}  (over $\mathbb C$), and Kleppe's result can be deduced from this.  Another reference for these results is \cite[Proposition 6.4]{kass08} (over an arbitrary algebraically closed field).

We can say much more about the local structure of the compactified Jacobian of a nodal curve.  One can give a complete description  using the theory of the presentation scheme as developed in \cite{altman90}, but for our purposes it is slightly more convenient to compute the local structure using deformation theory.  The following proposition is probably well-known (e.g., it can be found in \cite{kleppe81}), but we were unable to find a proof in print.

\begin{proposition} \label{locstrthm}
	Suppose that $X/k$ is an integral curve of genus $g$.  If $x$ is a point of $\bar{J}_{X/k}^d$ that corresponds to a rank $1$, torsion-free sheaf $I$ that fails to be locally free at $n$ nodes and no other points, then the  completion of the local ring of $\bar{J}^{d}_{X/k}$ at $x$ is isomorphic to 
\begin{displaymath}
	\left( \hat{\bigotimes}_{i=1}^{n} k[[u_i,v_i]]/(u_i v_i) \right)  \hat{\otimes} \left( \hat{\bigotimes}_{i=1}^{g-n} k[[w_i]] \right).
\end{displaymath}
Furthermore, the quotient of this local ring by the ideal $(u_1, v_1, \ldots, u_n, v_n)$  parametrizes those infinitesimal deformations that are locally trivial.
\end{proposition}
\begin{proof}
	Suppose that we are given a curve $X$ and a sheaf $I$ as in the hypothesis.  Let $F$ denote the deformation functor that parametrizes infinitesimal deformations of $I$.  This functor is pro-represented by the completion of the local ring of $
	\bar{J}_{X/k}^{d}$ at $x$, and we will prove the proposition by computing the functor $F$.  
	
	Suppose that the sheaf $I$ fails to be locally free precisely at the points $p_1, \ldots, p_n$.  For $j = 1, \ldots, n$, set $\hat{\mathscr{O}}_j$ equal to the completion of the local ring of $X$ at $p_j$ and $G_i$ equal to the functor that 
	parametrizes infinitesimal deformations of the sheaf $I|_{\hat{\mathscr{O}}_j}$ on $\hat{\mathscr{O}}_j$.
	
	An examination of the local-to-global spectral sequence computing the groups $\operatorname{Ext}^{*}(I,I)$ shows that the natural restriction map $$F \rightarrow G_1 \times \ldots \times G_n$$ is formally smooth.  Furthermore, the locally 
	trivial deformations of $I$ are precisely the deformations that map to the trivial deformations under the maps $F \rightarrow G_i$, $i= 1, \ldots, n$.  To complete the proof, it is enough to show that the ring $k[[u,v]]/(u v)$ is a miniversal 
	deformation ring for every $G_i$.
	
	The deformation functors $G_i$ are all abstractly isomorphic.  More precisely, the completion of the local ring of $X$ at $p_i$ is abstractly isomorphic to the algebra $\hat{\mathscr{O}} = k[[x,y]]/(x y)$.  Under this identification, the restriction of 
	$I$ to $\hat{\mathscr{O}}_{i}$ can be identified with the ideal $(x,y)$ (considered as an abstract module).  Each functor $G_i$ is isomorphic to the functor $G$ that parametrizes infinitesimal deformations of $(x,y)$ as a module over $k[[x,y]]/(xy)$.

	To show that $G$ has the desired form,  consider the ring $R= k[[u,v]]/(u v)$ and the $\hat{\mathscr{O}} \hat{\otimes} R$-module given by the ideal $\mathcal{I} = (x-u, y-v)$.  The pair $(R,\mathcal{I})$ defines a formal deformation of the module 
	$(x,y)$.  To complete the proof, we show that this formal deformation is miniversal.

	Rather than proving this claim by a direct computation, we will give a computation-free proof using the Abel map.  Consider first the special case of  the standard irreducible, nodal plane cubic $X_0/k$.  For this curve,  the natural map $X_0 
	\rightarrow \bar{J}^{-1}_{X_0/k}$ is an isomorphism  (\cite[Theorem 8.8]{altman80}), and this isomorphism identifies the ideal sheaf of the diagonal in $X_0 \times X_0$ with a Poincar\'{e} bundle on $X_0 \times \bar{J}_{X_0/k}^{-1}$.  In 
	particular, if we let $I_0$ denote the ideal sheaf of the node on $X_0$ then the associated global deformation functor $F$ is pro-represented by the ring $k[[u,v]]/(uv)$. For dimension reasons  the natural map $F \rightarrow G$ from global 
	deformations to local deformations is an isomorphism on tangent spaces, and so $F$ is a miniversal deformation of $G$.

	Thus we have proven the claim for the curve $X_0/k$.  But the functor $G$ is local and independent of the global curve under consideration, and thus claim for the functor $G$ in the  general case follows as well.  This completes the proof.
 \end{proof}

As a corollary, we can compute the multiplicity of $\bar{J}_{X/k}^{d}$ at a point.

\begin{corollary}\label{CorMultJac}
	Suppose that $X/k$ is an integral curve.  If $x$ is a point of $\bar{J}_{X/k}^{d}$ that corresponds to a rank $1$, torsion-free sheaf that fails to be locally free at $n$ nodes  and no other points, then we have that
\begin{displaymath}
	\operatorname{mult}_{x}\bar{J}_{X/k}^d = 2^n.
\end{displaymath}
\end{corollary}

\subsection{The Abel Map} \label{AbelSection}
Here we review the theory of the Abel map as developed in \cite{altman80}.  We only discuss the aspects of the theory that are relevant to our study of the theta divisor, and our discussion  is adequate for Gorenstein curves only.  In this paper, we work almost exclusively with curves satisfying the stronger condition that the singularities are planar.  A more satisfactory theory can be developed for curves with non-Gorenstein singularities, but there are some significant technical
complications.

The basic definition is the following.

\begin{definition}
	The {\bf Abel map} is the morphism $A: \operatorname{Hilb}_{X/k}^{d} \rightarrow \bar{J}_{X/k}^{2 g -2 -d}$ given by the rule:
	\begin{displaymath}
		Z \in \operatorname{Hilb}_{X/k}^{d}(T) \mapsto \mathcal{I}_{Z} \otimes \omega_{X/T} \in \bar{J}_{X/k}^{2g-2-d}(T),
	\end{displaymath}
	where $\mathcal{I}_{Z}$ is the ideal sheaf that defines the closed subscheme $Z$ and $\omega_{X/T}$ is the relative dualizing sheaf.
\end{definition}

This is dual to the convention used for non-singular curves, and it is necessary to adopt this convention in order to extend the Abel map to a map that is well-defined on all of $\operatorname{Hilb}_{X/k}^{d}$.

The Abel map realizes the Hilbert scheme as a (non-flat) family of projective spaces over the compactified Jacobian.  Given a point $x$ of $\bar{J}_{X/k}^{d}$ that corresponds to a sheaf $I$, the fiber $A^{-1}(x)$ is canonically isomorphic to the space $\mathbb{P}H^{1}(X,I)$.  This projective space should thus be thought of as a generalized linear system.  
This identification provides a geometric interpretation of the term $h^{0}(X,I)$ in Theorem \ref{teoC}.  By the Riemann-Roch formula, we have that $h^{0}(X,I)=h^{1}(X,I)$, and so the order of vanishing of $\Theta$ at $x$ is equal to $1$ more than the dimension of the projective space $A^{-1}(x)$.

While not explicitly stated, the identification of the fibers of $A$ follows from results in \cite[\S 4]{altman80}.  In that paper, the authors work with the map $\operatorname{Hilb}_{X/k}^{d} \rightarrow \bar{J}_{X/k}^{-d}$ given by $Z \mapsto I_{Z}$ and show that the fiber over a point corresponding to a sheaf $I$ can be identified with $\mathbb{P}(\operatorname{Hom}(I,\mathscr{O}_{X})^{\vee})$.  Our claim follows from duality theory.

Many of the modern proofs of the RST  establish the result by using the Abel map to reduce to a statement concerning the Hilbert scheme of points, a variety whose geometry is directly accessible (e.g.  \cite{kempf73}, \cite[Example 4.3.2]{fulton}).  For an integral, nodal curve $X$,  the Hilbert scheme $\operatorname{Hilb}_{X/k}^{d}$ is itself singular, and the presence of singularities makes it challenging to  fully generalize these arguments.  Similar difficulties were encountered in the paper of Smith-Varley \cite{smith2001}, concerning singularities of theta divisors for Prym varieties.  We do, however, use such an argument to prove Proposition \ref{AbelEvidenceProp}.  

In order to prove this result and others, we need to recall two facts about the Abel map.  Theorem 8.4 of \cite{altman80} states that if \ $d$ is an integer satisfying $d > 2 g -2$ and $X$ is Gorenstein, then the Abel map $\text{Hilb}^{d}_{X/k} \rightarrow \bar{J}_{X/k}^{2 g -2 -d}$ is a smooth fibration.   The Abel map can also be used to completely describe the compactified Jacobian of a genus $1$ curve.  If $X/k$ has genus $1$, then Theorem 8.8 of the same paper asserts that the Abel map $\text{Hilb}_{X/k}^{-1} = X \rightarrow \bar{J}^{-1}_{X/k}$ is an isomorphism.

\section{The Theta Divisor} \label{ThetaSection}
\subsection{Preliminaries}
There is an effective divisor on $\bar{J}^{g-1}_{X/k}$ that plays a role analogous to that of the classical theta divisor on a Jacobian.  For the compactifications studied in this paper, this divisor was constructed by Soucaris in \cite{soucaris} and by Esteves in \cite{esteves97}. We take the definition of theta divisor to be the following:
\begin{definition} \label{Def: Theta}
	Let $p \colon X \times \bar{J}_{X/k}^{g-1} \to \bar{J}_{X/k}^{g-1}$ denote the projection map.  The {\bf theta divisor} $\Theta$ is the $0$-th Fitting subscheme of $R^1p_{*}(\wp)$.
\end{definition}  

The sheaf $R^{1}p_{*}(\wp)$ depends on a particular choice of Poincar\'{e} bundle, but an application of the projection formula shows that the Fitting subscheme is independent of this choice.

We can derive a useful expression for $\Theta$ using the machinery of coherent cohomology.  The projection $p$ has relative dimension $1$ and $\wp$ is  $p$-flat, so we can conclude that there exists a 2-term complex of vector bundles $K^0 \stackrel{d}{\rightarrow} K^1$ that computes the higher direct images of $\wp$ ``universally."  That is, for any morphism $f \colon T \to \bar{J}_{X/k}^{g-1}$, we can form the Cartesian diagram
\begin{displaymath}
	\begin{CD}
		X \times T 		@>g>>		 	X \times \bar{J}_{X/k}^{g-1} 	\\
		@Vp_{T}VV					@VpVV 					\\
		T			@>f>> 					\bar{J}_{X/k}^{g-1}.
	\end{CD}
\end{displaymath}
The complex $(f^{*}K^{\cdot}, \operatorname{d}_{T})$ has the property that it computes the cohomology of $g^{*}(\wp)$ in the sense that $H^{i}(f^{*}K^{\cdot}) \cong R^{i}(p_{T})_{*}(g^{*}\wp)$ for all $i$.  This complex is not unique, but any two complexes
with this property are quasi-isomorphic.  

As the general degree $g-1$ line bundle has no cohomology, we have $p_{*}(\wp) = 0$ and $R^{1}p_{*}(\wp)$ is supported on a set of positive codimension.  In terms of the complex $(K^{\cdot},d)$, this translates into the facts that $\operatorname{rank}(K^0) = \operatorname{rank}(K^1)$ and $(K^{\cdot},d)$ is a locally free resolution of $R^1p_{*}(\wp)$. By definition, the divisor $\Theta$ is the vanishing locus of  $\det(d)$.  This is the desired expression for $\Theta$.

The theta divisor has reasonable algebro-geometric properties when $X$ is a curve with planar singularities.  It is proven in \cite[\S 3]{esteves97} and \cite{soucaris} that $\Theta$ is an ample Cartier divisor in $\bar{J}_{X/k}^{g-1}$ and, as an abstract scheme, it is an integral, local complete intersection variety.  Tracing through the construction of $\Theta$, we see that the points of $\Theta$ correspond to rank $1$, torsion-free sheaves $I$ with $h^{0}(X,I) \ne 0$.  

More generally, one can interpret the theta divisor in terms of the Abel map.  The theta divisor of a non-singular curve is equal to the image of the Abel map, and a similar statement holds for the theta divisor associated to the compactified Jacobian of a curve with at worst planar singularities.  We have defined an Abel map for Gorenstein curves with the property that the fiber over a point corresponding to a sheaf $I$ is equal to $\mathbb{P}H^{1}(X,I)$.  In particular, this fiber is non-empty if and only if $h^{1}(X,I)>0$ or, in other words, the point lies on the theta divisor.  This shows that the set-theoretic image of Abel map is equal to $\Theta$.  Since both $\Theta$ and $\text{Hilb}^{g-1}_{X/k}$ are reduced, the set-theoretic image of the Abel map with its reduced scheme structure, the scheme-theoretic image of the Abel map, and the theta divisor all coincide as closed sub-schemes of $\bar{J}_{X/k}^{g-1}$.  This argument was explained to the authors by Eduardo Esteves.

Although it will not be needed in this paper, in light of the discussion at the beginning of \S2, we point out that the problem of constructing a theta divisor for a (possible reducible) nodal curve has been  studied from the perspective of balanced line bundles on quasi-stable models (e.g. \cite[\S4]{caporaso07}, esp. Remark 2.4.2) and semi-stable sheaves  on stable curves, and degenerations of Abelian varieties (e.g. \cite[Lemma 3.8, Theorem 5.3]{alexeev02}).

\begin{remark}  \label{Remark: NonPlaneExample}
	In this paper, we have mostly avoided discussing the theta divisor of a curve with non-planar singularities 
	as the properties of this subscheme are not yet 
	well understood (but see Remark~\ref{Remark: NonPlaneExample}).  When $X$ is a curve with non-planar singularities, one can still use the formalism of determinants to construct a theta divisor, but there are two potential issues: the resulting subscheme $\Theta$ is not known to be a Cartier 
	divisor and the subscheme is not known to be equal to the image of the Abel map.  These issues do not arise if one considers only points corresponding to line bundles as in Kempf \cite{kempf70}.  We refer the interested reader to Soucaris \cite{soucaris} for a more detailed 
	discussion of these issues.  
\end{remark}

\subsection{Theta divisors and test arcs} \label{SubSection:TestArcsAndTheta}
 In this section, let $S=\operatorname{Spec} k[[t]]$ be the formal arc, with maximal ideal $(t)$ denoted by $0$.  For a map $S\to V$ from the arc to a scheme $V$, we denote the pull-back to $S$ of a Cartier divisor $D$ on $V$ by $D|_S$.
There are two results that allow us to use test arcs to study theta divisors.  The first is a standard result about the behavior of $\Theta$ with respect to the operation of restricting to a test arc.

\begin{proposition} \label{ThetaFuncProp}
	Suppose $X/k$ is an integral curve with at worst planar singularities and  let $S \rightarrow \bar{J}_{X/k}^{g-1}$ be an arc corresponding to an $S$-relatively rank $1$, torsion-free sheaf  $\mathcal{I}$ on $X \times S$, which is a deformation of the sheaf $\mathcal I_0=I$.  Assume that $H^{0}(X \times S,\mathcal{I})=0$.  Then $
	\Theta|_S = \operatorname{dim}_{k}H^{1}(X \times S, \mathcal{I}) \{ 0 \}$.  In other words, we have that
	\begin{displaymath}
		\operatorname{mult}_0\Theta|_S=
		\operatorname{dim}_{k}H^{1}(X \times S, \mathcal{I})\ge h^1(X,I).
	\end{displaymath}
	We recall the definition of the multiplicity at the beginning of \S\ref{MultiplicitySection}.
\end{proposition}

\begin{proof}
	Let $S \rightarrow \bar{J}_{X/k}^{g-1}$ be a given arc that corresponds to a sheaf $\mathcal{I}$ as in the hypothesis.  The fact that $\operatorname{dim}_{k}H^{1}(X \times S, \mathcal{I})\ge h^1(X,I)$ is clear from consideration of the sequence $0\to \mathcal I \stackrel{t}{\to} \mathcal I\to I\to 0$ (\cite[\S 1]{yano}).  The rest of the proposition follows from the general formalism of the determinant;  both the proof of Proposition 3.9 in \S5.3.2 of \cite{friedman} and the proof at the beginning of Section 1 in \cite{yano} generalize to our proposition.  

	We recall the proof for the sake of completeness.  Immediately after Definition \ref{Def: Theta} we outlined the construction of a particular 2-term complex $(K^{\cdot}, \operatorname{d})$ with the property that the zero locus of $\det \operatorname{d}$  is $\Theta$.  This complex is constructed so that it computes the cohomology of $\wp$.  One consequence of this fact is that the zero-th and first cohomology groups of the restricted complex $K_{S}^{0} \stackrel{\operatorname{d}_{S}}{\longrightarrow} K_{S}^{1}$ are $H^{0}(X \times S, \mathcal{I})$ and $H^{1}(X \times S, \mathcal{I})$ respectively.  On the other hand, the determinant $\det \operatorname{d}_{S}$ computes $\Theta|_{S}$ by functoriality.  We prove the proposition by investigating the map $\operatorname{d}_{S}$.
	
	 The vanishing of $H^{0}(X\times S, \mathcal{I})$, together with the Riemann-Roch Theorem and the Theorem on Base Change, implies that $H^{1}(X \times S, \mathcal{I})$ is a torsion $k[[t]]$-module.
	 Consequently, 
$$H^{1}(X \times S, \mathcal{I})
=k[[t]]/(t^{e_1}) \oplus \dots \oplus k[[t]]/(t^{e_m})$$
for some non-negative integers $e_1,\ldots,e_m$.	 
	As we are working over the spectrum of a power series ring, the modules $K^0_{S}$ and $K^1_S$ are free.  Thus, fixing suitible bases, we may represent  $\operatorname{d}_{S}$ as a matrix of the form
	\begin{displaymath}
		\left(
			\begin{matrix}
				t^{e_1}		&	0			& \dots	& 0 \\
				0			&	t^{e_2}		& \dots	& 0 \\
				\vdots		&	\vdots		& \ddots	& \vdots \\
				0			& 	0			& \dots	& t^{e_m}
			\end{matrix}
		\right).
	\end{displaymath}
	 The determinant of this matrix is $t^e$, with $e := e_1 + \dots + e_m$.  
	This is visibly the dimension of the vector space $H^{1}(X \times S, \mathcal{I})$.
\end{proof}

The next result states that there are  in fact test arcs  that achieve this lower bound.

\begin{proposition} \label{teotestarc}
Suppose $X/k$ is an integral curve with at worst planar singularities, and $x$ is a point of $\bar{J}_{X/k}^{g-1}$ that corresponds to a sheaf $I$.  Then there exists an arc $S\to \bar{J}_{X/k}^{g-1}$ with $0\mapsto x$ such that $
	\Theta|_S=h^1(X,I) \{ 0 \}$; i.e., 
$$	\operatorname{mult}_0\Theta|_S=h^1(X,I).$$

Furthermore, the arc can be chosen to correspond to a locally trivial family of sheaves.
\end{proposition}

\begin{proof}
The result follows from the previous proposition and the lemma below, which asserts the existence of  
an $S$-relatively rank $1$, torsion-free sheaf  $\mathcal{I}$ on $X \times S$  satisfying the assumptions of the proposition.
\end{proof}

\begin{lemma} \label{ArcExistLemma}
	Suppose that $X/k$ is an integral curve of genus $g$ with at worst planar singularities.  If $I$ is a degree $g-1$ rank 1, torsion-free sheaf on $X$, then there is a $S$-relatively rank 1, torsion-free sheaf $\mathcal{I}$ on $X \times S$ with the following properties:
	\begin{enumerate}

		\item $\mathcal{I}|_{X \times \{0\}}$ is isomorphic to $I$;
		
		\item $H^{0}(X \times S, \mathcal{I})=0$;
		
		\item $\operatorname{dim}_{k}(H^{1}(X \times S, \mathcal{I})) = h^{1}(X,I)$;
	
		\item As a deformation of $I$, the family $\mathcal{I}$ is locally trivial (Definition \ref{dfntriv}).
	\end{enumerate}
\end{lemma}
\begin{proof}
	The construction in  \cite[\S1.1]{yano} by Friedman and the first author applies to our situation without modification.  The proof of the lemma proceeds exactly as in  \cite[Theorem 1.9]{yano}, but  let us recall the construction and proof here for completeness.

	The desired deformation $\mathcal I$ is constructed from an auxiliary family of Cartier divisors.   We first construct a single Cartier divisor and then fit that divisor into a family over $S$.  By Lemma~\ref{Lemma: KillSections}, we may find a divisor $D$ that consists of $h^{0}(X, I)$ distinct points that lie in the non-singular locus of $X$ and satisfies the equations $H^{0}(X, I(-D))=0$, $H^{0}(X, I) = H^{0}(X, I(D))$.  Fix one such divisor $D$.
	
	The first order deformations of $D$ are classified by the cohomology group $H^{0}(X, \mathscr{O}_{D}(D))$.  Pick a first order deformation $D_1$ with the property that the corresponding section $\tau \in H^{0}(X, \mathscr{O}_{D}(D))$ does not vanish at any point in the support of $D$.  As the divisor is supported on the non-singular locus of $X$, we may extend $D_1$ to a deformation over $S$ (i.e. a $S$-relative Cartier divisor with central fiber equal to $D$).  Fix one such deformation $\mathcal{D}$.
	
	 Let $\mathcal{D}'$ denote the constant $S$-relative Cartier divisor with fiber $D$ and $\mathcal{I}'$ denote
the trivial $S$-relative rank 1, torsion-free sheaf with fiber $I$.  As in the proof of \cite[Theorem 1.9]{yano} we will show that the sheaf $\mathcal{I}$ given by 
\begin{displaymath}
	\mathcal{I} := \mathcal{I}'(\mathcal{D}-\mathcal{D}')
\end{displaymath}
	satisfies the desired conditions of the lemma.  Note that $\mathcal I$  is certainly a locally trivial deformation of $I$, so we need only show that it satisfies the cohomological conditions (2) and (3).

		Let us fix some notation.  Write $S_ n$ for $\operatorname{Spec}k[[t]]/(t^{n+1})$ and $X_n$ and $I_n$ for the obvious restrictions over $S_n$.  By virtue of the flatness of $\mathcal I$, we have exact sequences
	\begin{equation} \label{Eqn: SeqOnS}
		\begin{CD}
			0 	@>>>	\mathcal{I}	@>\cdot t^{n+1}>>	\mathcal{I}	@>>>	I_n	@>>>	0
		\end{CD}
	\end{equation}
	and
	\begin{equation} \label{Eqn: SeqOnSn}
		\begin{CD}
			0 	@>>>	I_{n-1}	@>\cdot t>>	I_n	@>>>	I_0	@>>>	0
		\end{CD}
	\end{equation}
	for all integers $n \ge 1$.  The right-hand map in \eqref{Eqn: SeqOnS} (resp. \eqref{Eqn: SeqOnSn}) is the map given by restriction to $S_n$ (resp. $S_0$), while the left-hand map is the inclusion of those local sections divisible by $t^{n+1}$ (resp. $t$).  We will write $\partial_n$ for the 
	connecting map associated to sequence \eqref{Eqn: SeqOnSn}.

We now claim that the restriction map 
\begin{equation}\label{eqnclaimcoh}
H^{0}(X_1, I_1) \to H^{0}(X,I)
\end{equation}
 is zero. 
The proof proceeds exactly as in \cite[Theorem 1.9]{yano}.
 Indeed, an inspection of the long exact sequence associated to \eqref{Eqn: SeqOnSn} (for $n=1$) shows that the  map $H^{0}(X_1, I_1) \to H^{0}(X,I)$ being zero is equivalent to the map $\partial_1 \colon H^{0}(X,I) \to H^{1}(X,I)$ being injective.  The map $\partial_1$ is given by the cup-product $\sigma \mapsto \partial_{D}(\tau \cup \sigma)$ (see for example \cite[Lemma 1.8]{yano}).  Here $\partial_{D}$ is the connecting map associated to the sequence $$0 \to I \to I(D) \to I|_{D}(D) \to 0.$$  As $H^{0}(X,I) = H^{0}(X,I(D))$, an inspection of the relevant long exact sequence allows us to conclude that $\partial_D$ is injective.  Similarly, the kernel of $\sigma \mapsto \tau \cup \sigma$ is $H^{0}(X, I(-D)) = 0$.  We have now proven that $\partial_1$ is injective, and hence that the map $H^{0}(X_1, I_1) \to H^{0}(X,I)$ is zero, establishing the claim.  
	
Statements (2) and (3) are a consequence of \eqref{eqnclaimcoh} being the zero map.  To establish (2), observe that the restriction maps
$H^{0}(X \times S, \mathcal{I}) \to H^{0}(X_0, I_0)$	
		 factor as 
$$
	 	H^{0}(X \times S, \mathcal{I}) \to H^{0}(X_1, I_1) \to H^{0}(X_0, I_0), $$
	and hence must also be zero.  By examining the long exact sequence associated to \eqref{Eqn: SeqOnS} (with $n=0$), we can conclude that $t \cdot H^{0}(X \times S, \mathcal{I}) = H^{0}(X \times S, \mathcal{I})$.  This is only possible if $H^{0}(X \times S, \mathcal{I}) = 0$, establishing (2) of the lemma as desired.
		
	The fact that Statement (3) holds is a consequence of \cite[Lemma 1.5, 1.6]{yano}.  Consider the group $H^{1}(X \times S, \mathcal{I})$.  The vanishing of $H^{0}(X\times S, \mathcal{I})$, together with a standard cohomology argument, implies that $H^{1}(X \times S, \mathcal{I})$ is a torsion $k[[t]]$-module.  Thus, considering the long exact sequence associated to \eqref{Eqn: SeqOnS}, it follows that 
 the co-boundary  map $H^{0}(X_{n}, I_{n}) \to H^{1}(X \times S, \mathcal{I})$ is an isomorphism for $n$ sufficiently large (i.e., we recover \cite[Lemma 1.5]{yano}).  Finally, considering the long exact sequences associated to \eqref{Eqn: SeqOnSn} for $n$ and $n-1$, an easy induction argument shows that the multiplication map $\cdot t : H^{0}(X_{n-1}, I_{n-1}) \to H^{0}(X_{n}, I_{n})$ is an isomorphism for all $n$.  Composing $n$ times, we see that $\cdot t^{n}: H^{0}(X_{0}, I_{0}) \to H^{0}(X_{n}, I_{n})$ is an isomorphism as well.  This proves (3), that $\dim_k H^1(X\times S,\mathcal I)=h^0(X,I)=h^1(X,I)$.
	\end{proof}

The property that an arc corresponds to a locally trivial family of sheaves implies that if $x$ lies in the singular locus of $\bar{J}_{X/k}^{g-1}$, then the  arc maps into the singular locus.  This property will be used in the proof of Theorem \ref{teoC}.

Since the locus of line bundles $J^{g-1}_{X/k} \subset \bar{J}_{X/k}^{g-1}$ is smooth, Theorem \ref{teoC} follows immediately from Propositions \ref{ThetaFuncProp}  and \ref{teotestarc} at all points corresponding to line bundles.  In fact, this statement remains true without any assumptions on the singularities of $X$; this is  a result due to  Kempf \cite{kempf70} using different methods.  In order to prove Theorem \ref{teoC} at points corresponding to sheaves that fail to be locally free, we must further examine the relationship between the multiplicity of a divisor on a singular variety at a point and  test arcs passing through that point.

\section{Multiplicity Theory}\label{MultiplicitySection}

  In order to compute the multiplicity of the theta divisor $\Theta$ at a point, we make use of some basic tools from multiplicity theory. 
   
\subsection{Preliminaries} 
To fix notation, suppose that we are given a locally Noetherian scheme $V/k$ and a point $x$ of $V$.  Set $\mathscr{O}$ equal to the local ring of $V$ at $x$ and $\mathfrak{m}$ equal to the maximal ideal of $\mathscr{O}$.  The {\bf tangent cone} to $V$ at $x$, written $\mathscr{T}_{x}(V)$, is defined to be the spectrum of the graded algebra $\operatorname{Gr}_{x}(V):= \oplus \mathfrak{m}^{n}/\mathfrak{m}^{n+1}$.  Similarly, we define the {\bf tangent space} $T_{x}(V)$ to be the spectrum of the symmetric algebra $\text{Sym}(\mathfrak{m}/\mathfrak{m}^2)$.  
The {\bf multiplicity}  of $V$ at $x$, written $\operatorname{mult}_xV$, is defined to be 
the degree of the tangent cone as a subscheme of the tangent space.

We will also be interested in the notion of the order of vanishing of a Cartier divisor:  if  $f$ is a non-zero element of $\mathscr{O}$, then the {\bf order of vanishing} of $f$, written $\operatorname{ord}(f)$, is the greatest integer $\nu$ such that $f \in \mathfrak{m}^{\nu}$.  The Krull Intersection Theorem implies that $\operatorname{ord}(f)$ is well-defined.  Furthermore, the order of vanishing is unchanged if we pass from $\mathscr{O}$ to the completion $\hat{\mathscr{O}}$.  Setting $\nu = \operatorname{ord}(f)$, the {\bf leading term} of $f$, written $f^{*}$, is defined to be the image of $f$ in  $\mathfrak{m}^{\nu}/\mathfrak{m}^{\nu+1}$.  Here we are considering $f^*$ as an element of the coordinate ring $\operatorname{Gr}_{x}(V)$ of the tangent cone to $\mathscr{O}$.    If $D$ is a Cartier divisor on $V$ that contains $x$, then we define the {\bf order of vanishing} of $D$ at $x$, written $\operatorname{ord}_{x} D$, to be $\operatorname{ord}(f)$ for $f\in \mathscr{O}_{V,x}$ a local equation for $D$. 

The first result that we need relates the multiplicity of a Cartier divisor to the order of vanishing of that divisor.  In short, one might expect that there are two ways for the  divisor to be singular:  either a local equation could vanish to high order or the ambient variety could itself be singular.
A standard result (e.g., \cite[Theorem 14.9]{matsumura}) gives such an inequality:  for a locally Noetherian scheme $V$, an effective Cartier divisor $D$ on $V$, and a point $x$ on $D$, we have
\begin{equation}\label{eqnmat}
		\operatorname{mult}_{x}D \ge    \operatorname{mult}_{x}V \cdot \operatorname{ord}_{x}D.
\end{equation}

In the case where $V$ is non-singular, it is well-known that equality holds in \eqref{eqnmat}, but this is no longer true when $V$ is singular.  The property that equality holds in \eqref{eqnmat} is closely
related to the property that the leading term $f^*$ of a local equation $f$ for $D$ is a non-zero divisor.  We illustrate this with an example.
\begin{example}
	Let $V = \operatorname{Spec}(k[x,y,z]/(x y))$ and $D$ be the divisor defined by $f = y - x^2$.  We have that $\operatorname{ord}_{0}(f)=1$ and $\operatorname{mult}_0(V) =2$.
However, $D$ is isomorphic to $\operatorname{Spec}k[x,z]/(x^3)$, which has multiplicity $\operatorname{mult}_0(D) =3$.  Thus we have a strict inequality $\operatorname{mult}_0 D >  \operatorname{mult}_0 V\cdot \operatorname{ord}_0(D) $.  
Note that the leading term $f^*$ is a zero-divisor.  If we write $X, Y, Z$ for the images of $x, y, z$ in $\mathfrak{m}/\mathfrak{m} ^2$, then the natural map $k[X, Y, Z]/(XY) \to  \operatorname{Gr}_{0}(V)$
	is an isomorphism. Under this isomorphism $f^*$ corresponds to $Y$, which is killed by $X$.  Furthermore, it can be shown that the ideal of $\mathscr{T}_0(D)$ in $\mathscr{T}_{0}(V)$ is 
	$(Y, X^3)$.  Note this ideal is not generated by $f^*$ alone.
\end{example}

\subsection{Relations with test arcs}
  In order to use the results of \S \ref{SubSection:TestArcsAndTheta} to prove further results about the local geometry of the theta divisor, we must explore the relationship between the multiplicity of  a Cartier divisor at a point and the restriction of the divisor to ``test arcs" through that point.   We use the notation from \S \ref{SubSection:TestArcsAndTheta} for test arcs and will assume for the remainder of the section that a given \emph{arc $S \to V$ does not factor through $D$}, so that $D|_{S}$ is a divisor on $S$.

It is not hard to check that if $S \mapsto V$ sends $0$ to $x$ and $V$ is non-singular at $x$, then $
\operatorname{mult}_x D \le \operatorname{mult}_0D|_S$, 
with  $\operatorname{mult}_x D=\operatorname{mult}_0D|_S$ if and only if $T_0S\nsubseteq \mathscr{T}_xD$.
When $V$ is singular at $x$, the situation is more complicated.  The analogous inequality is
\begin{equation}\label{eqnbasicarc}
	 \operatorname{ord}_x D \le \operatorname{mult}_0D|_S;
\end{equation}
however, when $V$ is singular, it is not always true that there exists an arc for which equality holds  (see Example \ref{exacusp}).  In contrast to the situation of \eqref{eqnmat}, the failure of  such arcs to exist
is not explained by the leading term $f^*$ being a zero-divisor: in Example \ref{exacusp} the element $f^*$ is a non-zero divisor and
equality holds in \eqref{eqnmat}.  

In those cases where   equality
holds in \eqref{eqnmat} and there exists an arc $S \to V$ satisfying $\operatorname{ord}_{x} D = \operatorname{mult}_{0} D|_{S} $,   it follows that 
\begin{equation} \label{eqnta}
	\operatorname{mult}_{x} D =\operatorname{mult}_{x} V \cdot \operatorname{mult}_{0} D|_S. 
\end{equation}
From the inequalities \eqref{eqnmat} and \eqref{eqnbasicarc} one might wonder whether there always existed arcs so that \eqref{eqnta} held.  Example \ref{exacusp} shows that this is not the case.
Thus our strategy for proving Theorem \ref{teoC} is to prove that for a special class of varieties there exist  arcs computing $\operatorname{ord}_{x} D$ and equality holds in \eqref{eqnmat}.

\begin{lemma}\label{lemluck}
Let $V/k$ be a scheme, $D$ an effective  Cartier divisor on $V$, and $x$ a point of $D$.  Assume that
we are given an isomorphism
\begin{equation} \label{eqnstmod}
\varphi:\hat {\mathscr O}_{V,x}\to 
	\mathscr{O}_{\text{std}} :=
	\left( \hat{\bigotimes}_{i=1}^{n} k[[u_i,v_i]]/(u_i v_i) \right)  \hat{\otimes} \left( \hat{\bigotimes}_{i=1}^{m} k[[w_i]] \right)
\end{equation} 
for some integers $n$ and $m$.  Here $\mathscr{O}_{\text{std}}$ denotes the ``standard" local model.  

Then: 
\begin{enumerate}
	\item There exists an arc $S \to V$ such that $\operatorname{ord}_{x} D = \operatorname{mult}_{0} D|_{S}$.
	\item Let $Z_{\text{std}}$ denote the closed subscheme of $\operatorname{Spec}(\mathscr{O}_{\text{std}})$
				defined by the ideal $(u_1, \dots, u_n, v_1, \dots, v_n)$ and $Z$ the subscheme
				of $\operatorname{Spec}(\hat{\mathscr{O}}_{V,x})$ that corresponds to $Z_{\text{std}}$ under
				$\phi$.  
If there exists an arc $S \to V$ with the following properties:
		\begin{enumerate}
			\item $\operatorname{ord}_{x} D = \operatorname{mult}_{0} D|_{S}$;
			\item the arc $S \to V$ factors as:
				\begin{equation*}
					\xymatrix{ 
					S \ar@{->}[r]\ar@{-->}[d]& V\\
					Z \ar@{->}[ru]&\\}, 
				\end{equation*}
		\end{enumerate}
		then equality holds in \eqref{eqnmat}.  In particular, there exists an arc $S \to V$ satisfying
			\begin{equation}\label{eqnluck}
				\operatorname{mult}_x D=\operatorname{mult}_xV
				\cdot \operatorname{mult}_0D|_S.
			\end{equation}
\end{enumerate}
In both $(1)$ and $(2)$, the distinguished arcs $S \to V$ are characterized by the property that $\operatorname{mult}_{0} D|_{S}$ is minimal.  

\end{lemma}

\begin{proof} Conceptually, the proof is straightforward.  We first pass from the scheme $V$ to the local ring $\mathscr{O}_{\text{std}}$.  The normalization of this ring is a product of power series ring.  We prove the lemma by computing the multiplicity and order of vanishing of $D$ using the normalization  map.  For clarity, we break the proof into several steps.

\noindent\textbf{Step 1} (Reduction to local algebra):  The multiplicity of $V$ and of $D$ at $x$ is unchanged if we pass from the schemes to their completed local rings.  Furthermore, any arc $S \mapsto V$ with the property that $0 \mapsto x$, factors through a map $S \rightarrow \operatorname{Spec}(\hat{ \mathscr{O}}_{V,x})$.  We can thus immediately reduce to proving the analogous lemma for the local ring $\mathscr{O}_{\text{std}}$.  For the remainder of the proof, we will work with this ring.  Pick an element $f \in \mathscr{O}_{\text{std}}$ that corresponds to a local equation for $D$.

\noindent\textbf{Step 2} (Normalization):
We now pass from $\mathscr{O}_{\text{std}}$ to the normalization of this ring.  Let us recall the standard presentation of this ring.  Let $\tilde{\mathscr{O}}_{\text{std}}$ to be the following ring:

\begin{displaymath}
	( \hat{\bigotimes}_{i=1}^{n} (k[[u_i]] \times k[[v_i]])  \hat{\otimes} ( \hat{\bigotimes}_{i=1}^{m} k[[w_i]] ).
\end{displaymath}

This ring is the normalization of $\mathscr{O}_{\text{std}}$, and we can describe the normalization map.
For each integer $i = 1, \ldots, n$, we have a ring homomorphism
$$
	\phi_{i}: k[[u_i, v_i]]/(u_i v_i) \rightarrow k[[u_{i}]] \times k[[v_{i}]] 
$$
given by the rule 
$$
	u_i \mapsto (u_i,0), v_i \mapsto (0,v_i).
$$

The normalization map  $\phi: \mathscr{O}_{\text{std}} \rightarrow \tilde{\mathscr{O}}_{\text{std}}$ is equal to the product $\phi = \phi_1 \hat{ \otimes } \ldots \hat{ \otimes } \phi_n \hat{ \otimes } \text{id}$.  Observe that $\tilde{\mathscr{O}}_{\text{std}}$ is a semi-local ring that is a product of power series rings and that the pre-image of the maximal idea of $\mathscr{O}_{\text{std}}$ is equal to the $2^n$ maximal ideals of $\tilde{\mathscr{O}}_{\text{std}}$.  Denote these ideals by $\mathfrak{n}_1, \dots, \mathfrak{n}_{2^n}$.  Finally, let us write $\tilde{f}$ for $\phi(f)$ and $\tilde{D}$ for the Cartier divisor on $\operatorname{Spec}(\tilde{\mathscr{O}}_{\text{std}})$ defined by $\tilde{f}$. We will now use the map $\phi$ to prove the lemma.  

\noindent \textbf{Step 3} (Proof of (1)):
    By inspection, we see that $\operatorname{ord}_{\mathfrak{m}} f = \operatorname{min}_{i}(\operatorname{ord}_{\mathfrak{n}_i} \tilde{f} )$.  Without loss of generality, we may assume that  $\operatorname{ord}_{\mathfrak{m}} f = \operatorname{ord}_{\mathfrak{n}_1} \tilde{f}$.  Using the fact that $\tilde{\mathscr{O}}_{\text{std}}$ is a product of power series rings, one may verify immediately that that there exists a  homomorphism $\psi: \tilde{\mathscr{O}}_{\text{std}} \rightarrow k[[t]]$ that maps $\mathfrak{n}_{i}$ into $(t)$ and has the property that $\operatorname{ord}_{\mathfrak{n}_{i}}( \tilde{f}) = \operatorname{ord}_{0}\psi( \tilde{f})$.  The arc that corresponds to the composition $\mathscr{O}_{\text{std}} \rightarrow \tilde{\mathscr{O}}_{\text{std}} \rightarrow k[[t]]$ satisfies the desired condition.

\noindent \textbf{Step 4} (Proof of (2)):
  Assume that there is an arc satisfying (a) and (b), and let $\mathscr{O}_{\text{std}} \to k[[t]]$ denote the corresponding ring map.
 From the Extension Theorem (e.g., see \cite{nagata}), we have that 
\begin{displaymath}
	\operatorname{mult}_{\mathfrak{m}}(D) = \sum_{i=1}^{2^n} \operatorname{mult}_{\mathfrak{n}_i}(\tilde{D}).
\end{displaymath}

On one hand, the ring $\tilde{\mathscr{O}}_{\text{std}}$ is a product of power series rings so the multiplicity of $\tilde{D}$ at any point is equal to the order of vanishing of $\tilde{D}$ at that point.   As $\tilde{D}$ is the pre-image of $D$, we have the bound $\operatorname{ord}_{\mathfrak{n}_i}(\tilde{D})\ge \operatorname{ord}_{\mathfrak{m}}(D)$.

On the other hand, we can use the hypothesis on arcs to produce an upper bound on the numbers $\operatorname{ord}_{\mathfrak{n}_i}(\tilde{D})$.  Let $S \to \operatorname{Spec}(\mathscr{O}_{\text{std}})$ satisfy conditions (a) and (b). Condition (b) ensures that, for $i = 1, \ldots, 2^n$, we can lift $S \to \operatorname{Spec}(\mathscr{O}_{\text{std}})$ to an arc $S \to \operatorname{Spec}(\tilde{\mathscr{O}}_{\text{std}})$ with $0 \mapsto \mathfrak{n}_i$.  The bound $\operatorname{ord}_{\mathfrak{n}_i}(\tilde{D}) \le \operatorname{ord}_{0}(\tilde{D}|_{S}) = \operatorname{ord}_{\mathfrak{m}}(D)$ follows from the existence of such an arc.  This establishes that (2) holds.

\noindent \textbf{Step 5} (Final Remark): Finally, we claim that the distinguished arcs in (1) and (2) are characterized by the property that $\operatorname{ord}_{0} D|_{S}$ is minimal
over all arcs $S \to V$ with $0 \mapsto x$.  This is immediate from \eqref{eqnbasicarc}.
\end{proof}

Lemma \ref{lemluck} is the final result which allows us to prove Theorem \ref{teoC} in the case of nodal curves.  If we modify the assumptions on the local structure of $V$, then the result becomes false.  We provide an example.

\begin{example}\label{exacusp}
	Let $V = \operatorname{Spec}(k[x,y,z]/(y^2-x^3))$ and $D$ be the Cartier divisor on $V$ defined by $f= x-z^{3}$.  For this choice of $V$ and $D$, the leading term
	$f^*$ is a non-zero divisor, and $\operatorname{mult}_0 D =  \operatorname{mult}_0 V\cdot \operatorname{ord}_0 D$.  However, there is no arc 
	$S \to V$ with the property that $\operatorname{ord}_0 D = \operatorname{mult}_0 D|_S$, and so, in particular, there is no arc such that  $\operatorname{mult}_0 D =  \operatorname{mult}_0 V\cdot \operatorname{mult}_0 D|_S$.
	Indeed, the multiplicities and orders are given by $\operatorname{mult}_0 D =2$, $\operatorname{ord}_0 D = 1$, and $\operatorname{mult}_0 V = 2$, while 
	an elementary argument shows that $\operatorname{mult}_0 D|_{S} \ge 2$ for all suitable arcs $S \to V$.
	  Note that the tangent cone of $V$ is also easily computed: there is an isomorphism of $\operatorname{Gr}_0(V)$ with $k[X,Y,Z]/(Y^2)$ such that $f^*=X$.  In particular, $f^*$ is not a zero divisor.	 We point out that the lower bound 
	 $\operatorname{ord}_0 D|_{S} \ge 2$ is sharp; equality holds for the arc defined by $x \mapsto t^2, y \mapsto t^3, z \mapsto 0$.  
\end{example}

\begin{remark}
Example \ref{exacusp} illustrates the difficulties that may arise in using the methods of this paper to establish \eqref{ConjA} for sheaves that fail to be locally free at singularities worse than nodes.
For instance, let $X/k$ be an integral curve of genus $g$ with a unique singular point, which is a cusp, and let $x \in \bar{J}_{X/k}^{d}$ correspond to a sheaf that fails to be locally free at the singular point.  In the notation of Example \ref{exacusp}, the completed local ring of $\bar{J}_{X/k}^{d}$ is isomorphic to a power series ring over the completed local ring of $V$ at $0$  (the proof is  similar to that of Proposition \ref{locstrthm}).
\end{remark}

\section{The proof of Theorem \ref{teoC}}
Now we prove Theorem \ref{teoC}.  We retain the notation for arcs that has been used in the previous two sections.  The idea of the proof is straightforward.  For a point $x$ of $\Theta$ corresponding to a sheaf $I$, Propositions \ref{ThetaFuncProp} and  \ref{teotestarc} ensure the existence of an arc $S$ through $x$ with minimal order of contact $h^1(X,I)$ with $\Theta$.  We then expect that $\operatorname{mult}_x\Theta=\operatorname{mult}_x\bar J^{g-1}_{X/k}\cdot h^1(X,I)$.  In order to check that this is in fact the case for sheaves failing to be locally  free only at nodes, we utilize Lemma \ref{lemluck}.

\begin{proof}[Proof of Theorem  \ref{teoC}]
	Suppose that $x$  is a point of $\Theta$ that corresponds to a rank $1$,  torsion-free sheaf $I$ that fails to be locally free at exactly $n$ points of $X$, all of which are nodes.  We need to compute both $\operatorname{mult}_x\Theta$ and $\operatorname{ord}_{x}\Theta$.  First observe that, by Proposition \ref{locstrthm}, the completion of the local ring of $\bar J^{g-1}_{X/k}$ at $x$ satisfies \eqref{eqnstmod} of  Lemma \ref{lemluck}.  We now aim to construct an arc $S\to \bar J^{g-1}_{X/k}$ satisfying conditions (2a) and (2b) of Lemma \ref{lemluck}.

Proposition \ref{teotestarc} states that there exists such an arc with  $\operatorname{mult}_0\Theta|_S= h^1(X,I)$.  Furthermore, every
arc $S \to \bar{J}^{g-1}_{X/k}$ with $0 \mapsto x$ must satisfy $\operatorname{mult}_0\Theta|_S\ge h^1(X,I)$ (Proposition \ref{ThetaFuncProp}). We can thus construct an arc which satisfies property (2a) of Lemma \ref{lemluck}; recall that such arcs are characterized by the fact that they minimize $\operatorname{mult}_0\Theta|_S$.
Furthermore, it follows from Lemma \ref{ArcExistLemma} (4) that $S \to \bar{J}_{X/k}^{g-1}$ can  be taken to correspond to a locally trivial family of rank $1$, torsion-free sheaves.  In particular, 
due to Proposition \ref{locstrthm}, such an arc satisfies condition (2b) of Lemma \ref{lemluck}.

Thus the conclusion of Lemma \ref{lemluck} holds, implying  that
$$
\operatorname{mult}_x\Theta=\operatorname{mult}_x \bar{J}_{X/k}^{g-1} \cdot \operatorname{mult}_0\Theta|_S=\left(
\operatorname{mult}_x \bar{J}_{X/k}^{g-1}\right)\cdot h^1(X,I),
$$
and $\operatorname{ord}_x\Theta=h^1(X,I)$.
We have already computed in Corollary \ref{CorMultJac} that $\operatorname{mult}_{x}\bar{J}_{X/k}^{g-1}$ is $2^n$, and so the theorem follows immediately.
\end{proof}

\section{Further Cases of Equation  (\ref{ConjA}) } \label{SectionPfA}
Here we prove several results that suggest that Equation (\ref{ConjA}) may  hold in greater generality.   Our proof of Theorem \ref{teoC} provides some inequalities on the order of vanishing and the 
multiplicity of the theta divisor of a curve with arbitrary planar singularities.  Without any assumptions, we can assert that
$\operatorname{mult}_{x}\Theta \ge  \left(\operatorname{mult}_x \bar{J}_{X/k}^{g-1}\right)\cdot \operatorname{ord}_x\Theta$; this is  \eqref{eqnmat}.  We also have an upper bound on $\operatorname{ord}_{x} \Theta$.  Indeed, Proposition \ref{teotestarc} is valid for any locally planar curve, so by \eqref{eqnbasicarc} we have $h^{0}(X,I) \ge \operatorname{ord}_{x} \Theta $.  We 
now discuss some results that can be proven using methods distinct from those used in the proof of Theorem \ref{teoC}.

 To begin, an immediate consequence of \eqref{ConjA} holding in general would be that the singular locus of $\Theta$ would be  equal to the union of the locus of line bundles with at least two linearly independent  global sections and the locus of sheaves that fail to be locally free.  We give a proof of this using results of Kempf \cite{kempf70} and  Kleppe \cite{kleppe81}.

\begin{proposition} \label{teoA}
	Suppose that $X/k$ is an integral curve with at worst planar singularities.  Define the following subsets of the theta divisor:
	\begin{enumerate}
		\item The set $W^1_{g-1}$ consisting of those points of $\Theta$ that correspond to sheaves $I$ with the property that $h^{0}(X,I) \ge 2$;
		\item The set $\partial \Theta$ consisting of those  points of $\Theta$ that correspond to sheaves that fail to be locally free.
	\end{enumerate}
Then we have that:
		\begin{displaymath}
			\Theta_{\operatorname{sing}} =  W^1_{g-1} \cup \partial \Theta,
					\end{displaymath}
and if $X$ is singular, then $\dim \Theta_{\operatorname{sing}}=g-2$.  More precisely, if nonempty, the set $\partial \Theta $ has an irreducible component of dimension $g-2$.
\end{proposition}
\begin{proof}
	Given an integral curve $X/k$ with at worst planar singularities, we intend to determine $\Theta_{\textnormal{sing}}$.  As we indicated earlier (\S\ref{remkleppekass}), the singular locus of the compactified Jacobian is the locus corresponding to sheaves that fail to be locally free.  Equation (\ref{eqnmat}) establishes the containment $\partial\Theta \subset \Theta_{\text{sing}}$.  We now show  $W^{1}_{g-1} \subset \Theta_{\text{sing}}$.  Let $x$ be a point of $W^{1}_{g-1}$.  If $x\in \partial \Theta$, then there is nothing to show.  Otherwise, $x$ corresponds to a line bundle and Theorem \ref{teoC} applies.

	We must also  show the reverse containment $\Theta_{\text{sing}} \subset W^{1}_{g-1} \cup \partial \Theta$.  Suppose that $x$ does not lie in $W^{1}_{g-1} \cup \partial \Theta $. By definition, $x$ then corresponds to a line bundle $L$ with $h^{1}(X,L)=1$.  A second application of Theorem \ref{teoC} shows that $x$ is not a singularity of $\Theta$.

	The final claim is that $\partial \Theta$ has a component of dimension $g-2$ when $X$ is singular.  We use the Abel map to reduce to the Hilbert scheme.  Fix a singularity $p$ of $X$ and consider the Zariski closure of the subset of points in  $\text{Hilb}^{g-1}_{X/k}$ that correspond to the union of $p$ and $g-2$ distinct smooth points of $X$.  It is immediate that this subset is ($g-2$)-dimensional and maps into $\partial \Theta$ under the Abel map.  To complete the proof, it is enough to show that the fiber over a general point of the image consists of a single point or, equivalently, that $h^{1}(X, \omega \otimes I_{Z})=h^0(X, \omega \otimes I_{Z})=1$ for $Z$ a closed subscheme consisting of the union of $p$ and $g-2$ general smooth points.  The latter statement follows from the known fact (Kleiman and Martins \cite[\S2]{kleiman09}) that the canonical bundle is generated by global sections.
\end{proof}

\begin{remark} \label{operatorname{mult}_0 D|_S}
	The proposition becomes false if we drop the assumption that $X$ has locally planar singularities.  Indeed, say $X$ is the non-Gorenstein
	genus $2$ curve with a unique singularity $p \in X$ analytically equivalent to the space triple point $k[[x,y,z]]/(xy, xz, yz)$.
	If $\omega$ is the dualizing sheaf and $q \in X$ is a point in the non-singular locus, then one can show that $\Theta$ is 
	non-singular at the point corresponding to $\omega(-q)$ (modify the proof of Lemma \ref{ThetaFuncProp}), and the sheaf $\omega(-q)$ fails to 
	be locally free.  The curve $X$ is non-Gorenstein, but we expect that Gorenstein examples exist.  Indeed, when $X$ has non-planar singularities,
	both $\bar{J}_{X/k}^{g-1}$ and $\text{Hilb}_{X/k}^{g-1}$  (and hence $\Theta$) are reducible by \cite{kleppe81a}; the analogue of Proposition \ref{teoA} thus fails
	provided some non-smoothable component of $\text{Hilb}^{g-1}_{X/k}$ is generically non-reduced and maps birationally onto its
	image under the Abel map.
\end{remark}

Some cases of Equation (\ref{ConjA}) can be established by using the Abel map to reduce to a statement about the Hilbert scheme.  The precise result that we prove is as follows:

\begin{proposition}\label{AbelEvidenceProp}
	Suppose that $X/k$ is an integral curve with at worst planar singularities.  Let $x$ be a point of the theta divisor $\Theta$ that corresponds to a sheaf $I$ satisfying $h^{0}(X,I)=1$.  Then we have that:
	\begin{displaymath}
		\operatorname{mult}_{x}\Theta = \operatorname{mult}_{x}\bar{J}_{X/k}^{g-1}
	\  \	\textnormal{ and} \ \ 
		\operatorname{ord}_x\Theta=1.
	\end{displaymath}
\end{proposition}
\begin{proof}
	Let $X/k$, $I$, and $x$ be given as in the hypothesis.  It is enough to prove that $\operatorname{mult}_x\Theta=\operatorname{mult}_x\bar{J}_{X/k}^{g-1}$.  Let $Z$ be the unique closed subscheme of $X$ such that $I = I_{Z} \otimes \omega$ and say that $z$ is the point of $\text{Hilb}_{X/k}^{g-1}$ that corresponds to $Z$.  By semi-continuity, there is a Zariski open neighborhood $V$ of $x$ in $\Theta$ such that, for all points $x_1$ in $V$, the corresponding sheaf $I_1$ satisfies $h^{0}(X,I_1)=1$.  Now consider the restriction of the modified Abel map $A: A^{-1}(V) \rightarrow V$.  It follows from \cite[\S4]{altman80} that the scheme-theoretic fibers of this map all consist of a single reduced point.  As $A: A^{-1}(V) \rightarrow V$ is a proper map of varieties of the same dimension, this map must be an isomorphism. 
	
	We have now reduced the proof to the problem of showing that the formula $\text{mult}_{z}(\text{Hilb}^{g-1}_{X/k}) = \text{mult}_{x}\bar{J}_{X/k}^{g-1}$ holds.  This equality is established by relating $\text{Hilb}_{X/k}^{g-1}$ to a Hilbert scheme of higher degree, which is easier to work with since the Abel map is then a smooth fibration.
	
	Pick a collection of distinct points $p_1, \ldots, p_g$ such that $h^{0}(X, \mathscr{O}_{X}(p_1 + \ldots + p_{g}))=1$ and the points are disjoint from both the singular locus of $X$ and the support of $Z$.  Set $Z_1$ equal to the closed subscheme defined by the ideal $I_{Z} \cap  I_{p_1} \cap \ldots \cap I_{p_g}$ and $z_1$ equal to the corresponding point of $\text{Hilb}_{X/k}^{2 g - 1}$.  Elementary deformation-theoretic considerations  show that there are compatible decompositions: 

$$
		T_{z_1}(\text{Hilb}^{2g-1}_{X/k}) = T_{z}(\text{Hilb}^{g-1}_{X/k}) \times \mathbb{A}^{g}, \\
$$
$$
		\mathscr{T}_{z_1}
		(\text{Hilb}^{2g-1}_{X/k}) = \mathscr{T}_{z}(\text{Hilb}^{g-1}_{X/k}) \times \mathbb{A}^{g}.$$

In fact, the decomposition exists on the level of completed local rings or, equivalently, the functors those rings pro-represent.  Briefly, the completed local ring of $\text{Hilb}^{2g-1}_{X}$ at $z_1$ pro-represents the functor that parametrizes infinitesimal deformations of the quotient map $\mathscr{O}_{X} \to \mathscr{O}_{Z_1}$.  The algebra $\mathscr{O}_{Z_1}$ decomposes as the product $\mathscr{O}_{Z} \times \mathscr{O}_{p_1} \times \dots \times \mathscr{O}_{p_g}$; infinitesimal deformations
of a quotient map into a product are in natural bijection with the products of infinitesimal deformations of the quotient maps into  the components, so the decomposition of $\mathscr{O}_Z$ induces
a decomposition of the completed local ring of $\text{Hilb}_{X/k}^{2 g -1}$.  The desired result now follows from the observation that infinitesimal deformations of $\mathscr{O}_{X} \to \mathscr{O}_{p_i}$ are parameterized by $k[[t]]$.

Given the decomposition, we obtain an equality of multiplicities: \[ \text{mult}_{z}(\text{Hilb}_{X/k}^{g-1}) = \text{mult}_{z_1}(\text{Hilb}_{X/k}^{2g-1}). \] Now the  Abel map $A: \text{Hilb}_{X/k}^{2g-1} \rightarrow \bar{J}_{X/k}^{-1}$ is a smooth fibration and so, setting $x_1 = A(z_1)$, we have a further equality of multiplicities:
\begin{displaymath}
	\text{mult}_{z_1}(\text{Hilb}_{X/k}^{2g-1}) = \text{mult}_{x_1}(\bar{J}_{X/k}^{-1}).
\end{displaymath}
Finally, the map given by translation by the line bundle $\omega(p_1+\ldots+p_g)$ defines an isomorphism $\bar{J}_{X/k}^{-1} \rightarrow \bar{J}_{X/k}^{g-1}$  sending $x_1$ to $x$.  In particular, we can conclude that $\text{mult}_{x_{1}}(\bar{J}_{X/k}^{-1}) = \text{mult}_{x}(\bar{J}^{g-1}_{X/k})$, and the proof is complete.
\end{proof}

\begin{remark}
	Using degeneration techniques it is possible  to prove that the equation $\operatorname{ord}_{x} \Theta  = h^{0}(X,I)$ holds 
in some cases not covered by Proposition \ref{AbelEvidenceProp} and Theorem \ref{teoC}.  Given a flat
family of curves $\mathscr{X}/S$ over $S = \operatorname{Spec}(k[[t]])$ such that the (geometric) generic fiber is non-singular and the special fiber $X$ has planar singularities,  
the compactified Jacobians of the fibers of this family  fit together to form a family $\bar{J}_{\mathscr{X}/S}$, and the appropriate theta divisors fit into
a divisor $\Theta_{S}$ in $\bar{J}_{\mathscr{X}/S}^{g-1}$. Let $I$ be a sheaf on $X$ corresponding to a point of $\Theta$.  If it is possible to fit $I$ into a family $\mathcal{I}$ of sheaves on $\mathscr{X}/S$ such that the dimensions of the cohomology groups are constant as a function of $s \in S$, then the semi-continuity of the order of vanishing together the Riemann Singularity Theorem for the generic fiber imply that $\operatorname{ord}_{x}(\Theta)=h^{0}(X,I)$.   We expect, however, that there are examples of $X$ and $I$ for which no such family can be found.
\end{remark}

\bibliography{bibl}

\providecommand{\bysame}{\leavevmode\hbox to3em{\hrulefill}\thinspace}
\providecommand{\MR}{\relax\ifhmode\unskip\space\fi MR }
\providecommand{\MRhref}[2]{%
  \href{http://www.ams.org/mathscinet-getitem?mr=#1}{#2}
}
\providecommand{\href}[2]{#2}
\begin{thebibliography}{10}

\bibitem{alexeev02}
Valery Alexeev, \emph{Complete moduli in the presence of semiabelian group
  action}, Ann. of Math. (2) \textbf{155} (2002), no.~3, 611--708.

\bibitem{alexeev04}
\bysame, \emph{Compactified {J}acobians and {T}orelli map}, Publ. Res. Inst.
  Math. Sci. \textbf{40} (2004), no.~4, 1241--1265.

\bibitem{Irred1977}
A.~B. Altman, A.~Iarrobino, and S.~L. Kleiman, \emph{Irreducibility of the
  compactified {J}acobian}, pp.~1--12, Sijthoff and Noordhoff, Alphen aan den
  Rijn, 1977.

\bibitem{altman80}
A.~B. Altman and S.~L. Kleiman, \emph{Compactifying the {P}icard scheme}, Adv.
  in Math. \textbf{35} (1980), no.~1, 50--112.

\bibitem{altman90}
\bysame, \emph{The presentation functor and the compactified {J}acobian}, The
  Grothen-dieck Festschrift, Vol.\ I, Progr. Math., vol.~86, Birkh\"auser
  Boston, Boston, MA, 1990, pp.~15--32.

\bibitem{ACGH}
E.~Arbarello, M.~Cornalba, P.~A. Griffiths, and J.~Harris, \emph{Geometry of
  algebraic curves. {V}ol. {I}}, Grundlehren der Mathematischen Wissenschaften
  [Fundamental Principles of Mathematical Sciences], vol. 267, Springer-Verlag,
  New York, 1985.

\bibitem{beauville77}
Arnaud Beauville, \emph{Prym varieties and the {S}chottky problem}, Invent.
  Math. \textbf{41} (1977), no.~2, 149--196.

\bibitem{bp}
U.~N. Bhosle and A.~J. Parameswaran, \emph{On the {P}oincar\'e formula and the
  {R}iemann singularity theorem over nodal curves}, Math. Ann. \textbf{342}
  (2008), no.~4, 885--902.

\bibitem{briancon}
J.~Brian{\c{c}}on, M.~Granger, and J.-P. Speder, \emph{Sur le sch\'ema de
  {H}ilbert d'une courbe plane}, Ann. Sci. \'Ecole Norm. Sup. (4) \textbf{14}
  (1981), no.~1, 1--25.

\bibitem{caporaso08}
Lucia Caporaso, \emph{N\'eron models and compactified {P}icard schemes over the
  moduli stack of stable curves}, Amer. J. Math. \textbf{130} (2008), no.~1,
  1--47.

\bibitem{caporaso07}
\bysame, \emph{Geometry of the theta divisor of a compactified {J}acobian}, J.
  Eur. Math. Soc. \textbf{11} (2009), 1385--1427.

\bibitem{yano}
Sebastian Casalaina-Martin and Robert Friedman, \emph{Cubic threefolds and
  abelian varieties of dimension five}, J. Algebraic Geom. \textbf{14} (2005),
  no.~2, 295--326.

\bibitem{dSouza}
Cyril D'Souza, \emph{Compactification of generalised {J}acobians}, Proc. Indian
  Acad. Sci. Sect. A Math. Sci. \textbf{88} (1979), no.~5, 419--457.

\bibitem{esteves97}
Eduardo Esteves, \emph{Very ampleness for theta on the compactified
  {J}acobian}, Math. Z. \textbf{226} (1997), no.~2, 181--191.

\bibitem{friedman}
Robert Friedman and John~W. Morgan, \emph{Smooth four-manifolds and complex
  surfaces}, Ergebnisse der Mathematik und ihrer Grenzgebiete (3) [Results in
  Mathematics and Related Areas (3)], vol.~27, Springer-Verlag, Berlin, 1994.

\bibitem{fulton}
William Fulton, \emph{Intersection theory}, second ed., Ergebnisse der
  Mathematik und ihrer Grenzgebiete. 3. Folge. A Series of Modern Surveys in
  Mathematics [Results in Mathematics and Related Areas. 3rd Series. A Series
  of Modern Surveys in Mathematics], vol.~2, Springer-Verlag, Berlin, 1998.

\bibitem{kass08}
Jesse Kass, \emph{Completions of {N}\'{e}ron models}, Ph.D. thesis, Harvard
  University, Cambridge, Massachusetts, 2009, available from
  http://www.proquest.com (publication number AAT 3365303).

\bibitem{kempf70}
George Kempf, \emph{The singularities of certain varieties in the {J}acobian of
  a curve}, Ph.D. thesis, Columbia University, New York, New York, 1970.

\bibitem{kempf73}
\bysame, \emph{On the geometry of a theorem of {R}iemann}, Ann. of Math. (2)
  \textbf{98} (1973), 178--185.

\bibitem{kleiman09}
Steven~Lawrence Kleiman and Renato~Vidal Martins, \emph{The canonical model of
  a singular curve}, Geom. Dedicata \textbf{139} (2009), 139--166.

\bibitem{kleppe81}
Hans Kleppe, \emph{The {P}icard scheme of a curve and its compactification},
  Ph.D. thesis, Massachusetts Institute of Technology, Cambridge,
  Massachusetts, 1981.

\bibitem{kleppe81a}
Hans Kleppe and Steven~L. Kleiman, \emph{Reducibility of the compactified
  {J}acobian}, Compositio Math. \textbf{43} (1981), no.~2, 277--280.

\bibitem{matsumura}
Hideyuki Matsumura, \emph{Commutative ring theory}, second ed., Cambridge
  Studies in Advanced Mathematics, vol.~8, Cambridge University Press,
  Cambridge, 1989, Translated from the Japanese by M. Reid.

\bibitem{melo07}
Margarida Melo, \emph{Compactified {P}icard stacks over $\overline {M}\sb g$},
  Math. Z. \textbf{263} (2009), no.~4, 939--957.

\bibitem{mk}
D.~Mumford, \emph{On the {K}odaira dimension of the {S}iegel modular variety},
  Algebraic geometry---open problems ({R}avello, 1982), Lecture Notes in Math.,
  vol. 997, Springer, Berlin, 1983, pp.~348--375.

\bibitem{nagata}
M.~Nagata, \emph{The theory of multiplicity in general local rings},
  Proceedings of the international symposium on algebraic number theory,
  {T}okyo \& {N}ikko, 1955 (Tokyo), Science Council of Japan, 1956,
  pp.~191--226.

\bibitem{pandharipande96}
Rahul Pandharipande, \emph{A compactification over {$\overline {M}\sb g$} of
  the universal moduli space of slope-semistable vector bundles}, J. Amer.
  Math. Soc. \textbf{9} (1996), no.~2, 425--471.

\bibitem{smith2001}
Roy Smith and Robert Varley, \emph{A {R}iemann singularities theorem for {P}rym
  theta divisors, with applications}, Pacific J. Math. \textbf{201} (2001),
  no.~2, 479--509.

\bibitem{soucaris}
A.~Soucaris, \emph{The ampleness of the theta divisor on the compactified
  {J}acobian of a proper and integral curve}, Compositio Math. \textbf{93}
  (1994), no.~3, 231--242.

\end{thebibliography}
\end{document}